\documentclass[11pt]{amsart}

\usepackage[utf8]{inputenc}
\usepackage[margin=1in]{geometry}
\usepackage[titletoc,title]{appendix}
\usepackage{pifont}
\RequirePackage[dvipsnames,usenames]{color}

\usepackage{amsmath,amsfonts,amssymb,mathtools}
\usepackage{enumitem}
\usepackage{graphicx,float}

\usepackage[ruled,vlined]{algorithm2e}
\usepackage{algorithmic}
\usepackage{amsthm}

\usepackage{ytableau}
\usepackage{comment}
\usepackage{bigints}

\newtheorem{defn}{Definition}[section]
\newtheorem{theorem}{Theorem}[section]
\newtheorem{corollary}[theorem]{Corollary}
\newtheorem{lemma}[theorem]{Lemma}
\newtheorem{prop}[theorem]{Proposition}

\newtheorem{example}[theorem]{Example}

\newcommand{\frakm}{\mathfrak{m}}

\DeclareMathOperator{\Hom}{Hom}
\DeclareMathOperator{\Ext}{Ext}

\DeclareMathOperator{\Dim}{dim}

\makeatletter
\@namedef{subjclassname@2020}{%
  \textup{2020} Mathematics Subject Classification}
\makeatother

\title{A Note on the Multiplicites of The Determinantal Thickenings of Maximal Minors}
\author{Jiamin Li}
\address{Department of Mathematics, Statistics, and Computer Science, University of Illinois at Chicago,
Chicago, IL 60607}
\email{jli283@uic.edu}
\subjclass[2020]{}

\begin{document}

\maketitle

\begin{abstract}
	Let $S=\mathbb{C}[x_{ij}]_{1\leq i \leq m, 1\leq j\leq n}$ be a polynomial ring of $m\times n$ variables over $\mathbb{C}$ and let $I$ be the determinantal ideal of maximal minors of $S$. Using the representation theoretic techniques introduced in the work of Raicu, Weyman and Witt, we prove the existence of the generalized $j$-multiplicities $\epsilon^j(I)$ defined by Dao and Monta\~{n}o in \cite{DaoMontano19}. We will also give a closed formula of $\epsilon^j(I)$, which generalized the results in the work by Jeffries, Monta\~{n}o and Varbaro the work by Kenkel in the maximal minors case.
\end{abstract}

\section{Introduction}
The Hilbert-Samuel multiplicity (see \cite[Ch4]{BrunsHerzog} for more detailed discussion) has played an important role in the study of commutative algebra and algebraic geometry. The attempt of its generalization can be traced back to the work of Buchsbaum and Rim \cite{BuchsbaumRim} in 1964. One of the more recent generalizations is defined via the $0$-th local cohomology (see for example \cite{KatzValidashti}). 

In \cite{JeffriesMontanoVarbaro} the authors studied this generalized multiplicities of several classical varieties, in particular they calculated the formula of the generalized multiplicities of determinantal ideals of non-maximal minors. A further generalized multiplicity is defined in \cite{DaoMontano19} via the local cohomology of other indices, which is necessary in some situations, e.g. the determinantal ideals of maximal minors. However, the existence of such multiplicity is not known in general. In the unpublised work \cite{Kenkel} the author calculated the closed formula, and thus showed the existence, of the generalized $j$-multiplicity defined in \cite{DaoMontano19} of determinantal ideals of maximal minors of $m\times 2$ matrices. The goal of this note is to generalize the aforementioned results to the determinantal ideals of maximal minors of $m \times n$ matrices for any $m>n>1$.

We first recall the definition and the result of the multiplicities of the $0$-th local cohomology from \cite{JeffriesMontanoVarbaro}.
\begin{defn}\label{defn_dm}(see \cite{DaoMontano19} for more details)
	Let $S$ be a Noetherian ring of dimension $k$ and $\frakm$ a maximal ideal of $S$. Let $J$ be an ideal of $S$. Define
$$\epsilon_+^j(J) := \limsup_{D\rightarrow\infty}\dfrac{k!\ell(H^j_\frakm(S/J^D))}{D^k}.$$
	Suppose $\ell(H^j_\frakm(S/J^D)) < \infty$, then we define
	$$\epsilon^j(J) := \lim_{D \rightarrow \infty} \dfrac{k!\ell(H^j_\frakm(S/J^D))}{D^{k}}$$ if the limit exist. 
\end{defn}
One of the results by Jeffries, Monta\~{n}o and Varbaro in their paper is the following.
\begin{theorem}(See \cite{JeffriesMontanoVarbaro})\label{0_multiplicities_result}
	Let $I_p$ be the determinantal ideal of $p \times p$-minors of $S$ where $S$ is a polynomial ring of $m \times n$ variables over $\mathbb{C}$ and $0<p<n\leq m$. Then $$\epsilon^0(I_p) = cmn\bigintss_{\Delta}(z_1...z_n)^{m-n}\prod_{1\leq i\leq j\leq m}(z_j-z_i)^2dz$$ where $\Delta=\operatorname{max}_i\{{z_i}+t-1\leq \sum z_i \leq t\}$. 
\end{theorem}
However, when $p=n$, $H^0_\frakm(S/I_n^D)$ is always $0$. To avoid this triviality we will instead study the multiplicites of $I_n$ of other cohomological indices.

It was proved in \cite{DaoMontano19} that when $S$ is a polynomial ring of $k$ variables and when $J$ is a homogeneous ideal of $S$, we have for all $\alpha \in \mathbb{Z}$, 
$$\limsup_{D\rightarrow\infty}\dfrac{k!\ell(H^j_\frakm(S/J^D)_{\geq \alpha D})}{D^k} < \infty.$$
As a corollary of the above result, combined with the result from \cite{Raicu}, which states that if $S$ is a polynomial ring of $m \times n$ variables and $I_p$ is a determinantal ideal of $p \times p$-minors of $S$, then $H^i_\frakm(S/I_p^D)_j = 0$ for $i\leq m+n-2$ and $j<0$, we get that $\epsilon_+^j(I_p) < \infty$ for $j\leq m+n-2$ (see \cite[Ch 5]{DaoMontano19}). Note that, as mentioned in \cite[Ch 7]{DaoMontano19}, even if $\epsilon^j(J)$ exists, it doesn't have to be rational (see the example in \cite[Ch 3]{CutkoskyHST}). Therefore it is natural to ask about the existence, and if exists, the rationality of the limit $\epsilon^j(J)$.

In this note we will exploit the representation theoretic techniques to specifically study the multiplicities of $I_n$. More precisely, we show the following.

\begin{theorem}(Theorem \ref{existence_limit}\label{main_thm})
	\begin{enumerate}
		\item If $j\neq n^2-1$, then $\ell(H^j_\frakm(R/I_n^D))$ is either $0$ or $\infty$. 
		\item If $j=n^2-1$, then $\ell(H^j_\frakm(R/I_n^D)) < \infty$ and is nonzero. Moreover we have $$\epsilon^j(I_n) = (mn)!\prod^{n-1}_{i=0}\dfrac{i!}{(m+i)!},$$ which is rational.
	\end{enumerate}
\end{theorem}
Therefore we see that the only interesting cohomological index to our question is $n^2-1$, and we solve the problem of calculating the generalized $j$-multiplicites of determinantal ideals of maximal minors completely.

\medskip
\textbf{Notations.}
In this paper $\ell(M)$ will denote the length of a module $M$, $S$ will always denote the polynomial ring $\mathbb{C}[x_{ij}]$ of $m\times n$ variables and $I_p$ the determinantal ideal of $p \times p$ minors of $S$. All rings are assumed to be unital commutative.

\section{Preliminaries on Schur Functor}
We will recall the basic construction of the Schur functors, more information can be found in \cite{FultonHarris} and \cite{Weyman}. Let $V$ be an $n$-dimensional vector space over $\mathbb{C}$. We define the dominant weight of $V$ to be $\lambda = (\lambda_1,...,\lambda_n) \in \mathbb{Z}^n$ such that $\lambda_1\geq ... \geq \lambda_n$. Note that $(\lambda_1,\lambda_2,0,0,...,0) = (\lambda_1,\lambda_2)$. Furthermore we denote $(c,...,c)$ by $(c^n)$. We say $\lambda=(\lambda_1,\lambda_2,...) \geq \alpha = (\alpha_1,\alpha_2,...)$ if each $\lambda_i \geq \alpha_i$. Given a weight we can define an associated Young diagram with numbers filled in. For example if $\lambda=(3,2,1)= (3,2,1,0,0,0)\in \mathbb{Z}^6$, then we can draw the Young diagram 
\[
\begin{ytableau}
	1 & 2 & 3\\
	4 & 5\\
	6\\	
\end{ytableau}
\]
Let $\mathfrak{S}_n$ be the permutation group of $n$ elements. Let $P_\lambda = \{g\in\mathfrak{S}_n:g \text{ preserves each row}\}$ and $Q_\lambda=\{g\in\mathfrak{S}_n:g \text{ preserves each column}\}$. Then we define $a_\lambda = \sum_{g\in P_\lambda}e_g$, $b_\lambda = \sum_{g\in Q_\lambda}\operatorname{sgn}(g)e_g$, and moreover $c_\lambda  = a_\lambda \cdot b_\lambda$. 

Recall that the Schur functor $S_\lambda(-)$ is defined to $$S_\lambda(V) = \operatorname{Im}(c_\lambda\big|_{V^{\otimes d}}).$$ 

Let $V$ be an $n$-dimensional $\mathbb{C}$-vector space. We have a formula for the dimension of $S_\lambda V$ as $\mathbb{C}$-vector space.
\begin{prop}(See \cite[Ch2]{FultonHarris})\label{dim_schur}
	Suppose $\lambda = (\lambda_1,...,\lambda_n) \in \mathbb{Z}^n_\text{dom}$. Then we have $$\Dim(S_\lambda V) = \prod_{1\leq i < j \leq n}\dfrac{\lambda_i-\lambda_j+j-i}{j-i}.$$
\end{prop}
From the formula of $\Dim(S_\lambda V)$ it is easy to see the following.
\begin{corollary}\label{same_dim}
	For any $c\in \mathbb{N}$ we have
	$$\Dim(S_\lambda V) = \Dim(S_{\lambda+(c^n)}V).$$
\end{corollary}
\section{$\operatorname{GL}$-equivariant vector space decompositions of $\Ext^j_S(S/I_p^d)$}
In this section we recall the $\operatorname{GL}$-equivariant $\mathbb{C}$-vector spaces decompositions of $\Ext^j_S(S/I_p^D)$ given in \cite{Raicu}. This will be the key ingredient in the disuccsion of multiplicities in section 4. 
First recall that we have the $\operatorname{GL}$-equivariant decomposition (Cauchy formula) of $S$:$$S=\bigoplus_{\lambda\in \mathbb{Z}^{\text{dom}}_{\geq 0}} S_\lambda\mathbb{C}^m \otimes S_\lambda\mathbb{C}^n.$$

It was shown in \cite{DeConciniEisenbudProcesi} that a $\operatorname{GL}$-equivariant ideal $I$ of $S$ can be written as $$I_\lambda = \bigoplus_{\mu\geq \lambda}S_\mu\mathbb{C}^m \otimes S_\mu\mathbb{C}^n,$$ and in particular the ideal of $p\times p$ minors is equal to $I_{(1^p)}$.  

\begin{theorem}\label{schur-decompos}(See \cite[Theorem 3.3]{RaicuWeyman}, \cite[Theorem 2.5, Theorem 3.2, Lemma 5.3]{Raicu})
	There exists a $GL$-equivariant filtration of $S/I_p^d$ with factors $J_{\underline{z},l}$ which are quotients of $I_{\underline{z}}$. Therefore we have the following vector spaces decomposition of $\Ext^j_S(S/I^d_p,S)$:
	\begin{align}\label{decompose_first}
		\Ext_S^j(S/I_p^d,S) = \bigoplus_{(\underline{z},l) \in \mathcal{Z}^d_p} \Ext_S^j(J_{\underline{z},l},S)
	\end{align}
	and recall that $\mathcal{Z}^d_p$ is given by
\begin{align*}
    \mathcal{Z}^d_p=\big\{(\underline{z},l): 0\leq l \leq p-1, \underline{z}\in \mathcal{P}_n, z_1=...=z_{l+1}\leq d-1&, \\|\underline{z}| + (d-z_1)\cdot l +1 \leq p\cdot d \leq |\underline{z}|+(d-z_1)\cdot(l+1)\big\}
\end{align*}
and we have
\begin{align}
\operatorname{Ext}_S^j(J_{(\underline{z},l)},S) = \bigoplus_{\substack{0\leq s \leq t_1 \leq ... \leq t_{n-l}\leq l\\ mn -l^2 -s(m-n)-2(\sum^{n-l}_{i=1}t_i)=j \\ \lambda \in W(\underline{z},l;\underline{t},s)}} S_{\lambda(s)}\mathbb{C}^m \otimes S_\lambda\mathbb{C}^n
\end{align}
where $\mathcal{P}_n$ is the collection of partitions with at most $n$ nonzero parts, which means $z_1\geq z_2 \geq ... \geq z_n \geq 0$. Moreover the set $W(\underline{z},l,\underline{t},s)$ consists of dominant weights satisfies the following conditions:
\begin{align}\label{res_weight_gen}
\begin{cases}
\lambda_n \geq l -z_l -m, \\
\lambda_{t_i+i} = t_i-z_{n+1-i}-m & i=1,...,n-l,\\
\lambda_s \geq s-n \text{ and } \lambda_{s+1} \leq s-m.
\end{cases}
\end{align}
and the $\lambda(s)$ is given by $$\lambda(s)=(\lambda_1,...,\lambda_s,(s-n)^{m-n},\lambda_{s+1}+(m-n),...,\lambda_n+(m-n)) \in \mathbb{Z}^m_{\text{dom}}.$$
	In fact in our case we have $\lambda_n=l-z_l-m$. This also implies that $t_{n-l}=l$.
\end{theorem}

In the rest of the paper we will assume $p=n$, i.e. we only focus on the maximal minors case.

\begin{lemma}\label{weights_max_minors}
	In Theorem \ref{schur-decompos} we have $l=n-1$. Therefore the pair $(\underline{z},l)$ in Theorem \ref{schur-decompos} is of the form $((c)^n,n-1)$ for $c\leq d-1$. In particular we have $((d-1)^n,n-1)$ in $\mathcal{Z}^d_n$.
\end{lemma}
\begin{proof}: Note the restriction $l\leq p-1$ gives $l\leq n-1$. It is easy to check that $((d-1)^n,n-1)$ is in $\mathcal{Z}^d_n$. On the other hand, assume that there exists $(\underline{z},l)$ in $\mathcal{Z}^d_n$ such that $l\leq n-2$. From Theorem \ref{schur-decompos} we have the restriction $$|\underline{z}|+(d-z_1)\cdot(l+1) \geq nd$$ when $p=n$. However by our assumption we have
\begin{align*}
    |\underline{z}|+(d-z_1)(l+1) &\leq |\underline{z}|+(d-z_1)(n-1)  \\&=  |\underline{z}|+d(n-1)-z_1(n-1) \\&\leq  nz_1+d(n-1)-z_1(n-1) \\&=   z_1+d(n-1) \\&\leq  d-1 +d(n-1) = nd - 1 < nd.
\end{align*}
Contradicting to our restriction. Therefore we must have $l=n-1$. Moreover, by the definition of $(\underline{z},l)$ we have $z_1=...=z_{l+1}$, therefore in our case we have $z_1=...=z_n$. So the $(\underline{z},l)$ is of the form $((c)^n,n-1)$ for $c\leq d-1$.
\end{proof} 
In the rest of the paper we will denote $I:=I_n$.
Using this information we can also gives a criterion of the vanishing of the $\Ext$ modules. Recall that the highest non-vanishing cohomological degree of $S/I^d_n$ is $n(m-n)+1$ (see \cite{Huneke}). This can be seen from the following lemma as well.
\begin{lemma}\label{vanishing_degree}
	In our setting $\Ext^j_S(S/I^d,S)\neq 0$ if and only if $m-n$ divides $1-j$ and $j\geq 2$. Moreover, $\Ext^{n(m-n)+1}_S(S/I^d,S)\neq 0$ if and only if $d\geq n$.
\end{lemma}
\begin{proof}
	By Lemma \ref{weights_max_minors}, the weights $\lambda\in W:=W(\underline{z},n-1,(n-1),s)$ have the restrictions
	\begin{align}\label{res_weight_max}
		\begin{cases}
		\lambda_n=n-1-z_{n-1}-m,\\
		\lambda_s\geq s-n \text{ and } \lambda_{s+1}\leq s-m.
		\end{cases}
	\end{align}
	We also have 
	\begin{align}\label{important_eq}
		mn-(n-1)^2-s(m-n)-2(n-1)=j \implies s(m-n)=n(m-n)+1-j.
	\end{align}
	By Theorem \ref{schur-decompos}, $\Ext^j_S(S/I^d_n,S)\neq 0$ if and only if the set $W$ is not empty, then by (\ref{res_weight_max}) and (\ref{important_eq}) this means $m-n$ divides $n(m-n)+1-j \implies m-n$ divides $1-j$ and $s=(n(m-n)+1-j)/(m-n) \leq l = n-1 \implies j\geq 2$. This proves the first statement.

	To see the second statement, note that when $j=n(m-n)+1 = mn-(n-1)^2-2(n-1)$ we have $s=0$. In this case we have the restriction
	\begin{align*}
		\begin{cases}
			\lambda_n = n-1-z_{n-1}-m,\\
			\lambda_1 \leq -m.
		\end{cases}
	\end{align*}
	If $d<n$, then $\lambda_n \geq n-d-m > -m \geq \lambda_1$, a contradiction, so that means the set $W$ is empty. On the other hand if $d\geq n$ then $W$ is not empty. So $\Ext^{n(m-n)+1}_S(S/I^d,S) \neq 0$ if and only if $d\geq n$. 
\end{proof}

In our proof of the main theorem, we will need an important property of the $\Ext$-modules, which only holds for maximal minors.

\begin{prop}\label{isom_ext_modules}(See \cite[Corollary 4.4]{RaicuWeymanWitt})
	We have $\Hom_S(I^d,S)=S$, $\Ext^1(S/I^d,S)=0$ and $\Ext^{j+1}_S(S/I^d,S) = \Ext^j_S(I^d,S)$ for $j>0$.
\end{prop}

\begin{lemma}(See \cite[Theorem 4.5]{RaicuWeymanWitt})\label{injectivity}
	Given the short exact sequence $0\rightarrow I^{d}\ \rightarrow I^{d-1} \rightarrow I^{d-1}/I^d \rightarrow 0$, the induced map $\Ext^j_S(I^{d-1},S)\rightarrow \Ext^j_S(I^{d},S)$ is injective for any $j$ such that $\Ext^j_S(I^d,S)\neq 0$.
\end{lemma}

In order to prove our main theorem, we need to investigate the length of the $\Ext$-modules. We will need the following fact.

\begin{lemma}\label{length=dim}
	Given a graded $S$-module $M$ we have $\ell(M) = \Dim_\mathbb{C}(M)$.
\end{lemma}
\begin{proof}
	First assume $M$ is finitely graded over $\mathbb{C}$ and write $M=\oplus_i^\alpha M_i$. We will use the $\mathbb{C}$-vector space basis of each $M_i$ to construct the composition series of $M$ over $S$. Suppose $M_\alpha = \operatorname{span}(x_1,...,x_r)$ and consider the series $$0\subsetneq \operatorname{span}(x_1) \subsetneq \operatorname{span}(x_1,x_2) \subsetneq ... \subsetneq \operatorname{span}(x_1,...x_r)=M_\alpha.$$ Note that each $x_i$ can be annihilated by the maximal ideal $\frakm$ of $S$ since multiplying $x_i$ with elements in $\frakm$ will increase the degree. Since $Sx_i$ is cyclic, we have $Sx_i \cong S/\frakm$. Therefore each quotient of the above series is isomorphic to $S/\frakm$, so the series above is a composition series. Repeat this procedure for each graded piece of $M$ we get a composition series of $M$ and that $\ell(M) = \Dim_\mathbb{C}(M)$. 

	On the other hand if $M$ has infinitely many graded pieces over $\mathbb{C}$ so that $\Dim_\mathbb{C}(M) = \infty$, then the above argument shows that we can form a composition series of infinite length, and so $\ell(M)=\infty$.
\end{proof}

\begin{prop}\label{length_ext_fin}
	In our setting $\ell(\Ext^j_S(S/I^d,S)) < \infty$ and is nonzero if and only if $j=n(m-n)+1$ which corresponds to $s=0$ in Theorem \ref{schur-decompos}, and $d\geq n$. 
\end{prop}

\begin{proof}
	The correspondence of the cohomological index and $s$ can be seem in the proof of Lemma \ref{vanishing_degree}, and the condition $d\geq n$ can be seen from Lemma \ref{vanishing_degree} as well. Observe that the decomposition (\ref{decompose_first}) is finite, so we need to consider the decomposition of each $\Ext^j_S(J_{(\underline{z},l)},S)$. Suppose $s=0$. Then we have the restriction 
	\begin{align*}
		\begin{cases}
			\lambda_n=n-1-{z_{n-1}}-m,\\
			\lambda_1 \leq -m.
		\end{cases}
	\end{align*}
	Therefore in this case the set $W(\underline{z},n-1,(n-1),0)$ is bounded above by $(-m,...,-m,n-1-z_{n-1}-m)$ and below by $(-m,n-1-z_{n-1}-m,...,,n-1-z_{n-1}-m)$ and so is a finite set. Thus $\Ext^j_S(J_{(\underline{z},l)},S)$ can be decomposed as a finite direct sum of $S_{\lambda(s)}\mathbb{C}^m \otimes S_\lambda\mathbb{C}^n$ for $\lambda \in W(\underline{z},n-1,(n-1),0)$. By Proposition \ref{dim_schur} it is clear that the dimension of each $S_{\lambda(s)}\mathbb{C}^m \otimes S_\lambda\mathbb{C}^n$ is finite. So by Lemma \ref{length=dim}, $\ell(\Ext^j_S(J_{(\underline{z},l)},S)) = \Dim_\mathbb{C}(\Ext^j_S(J_{(\underline{z},l)},S)) < \infty$.

	On the other hand suppose $s\neq 0$. Then we have the restriction 
	\begin{align*}
		\begin{cases}
			\lambda_n=n-1-{z_{n-1}}-m,\\
			\lambda_s \geq s-n, \lambda_{s+1} \leq s-m.
		\end{cases}
	\end{align*}
	Since $\lambda_s \geq s-n$ implies that any weight that is greater than $(s-n,...,s-n,s-m,,...,s-m,n-1-z_{n-1}-m)$ is in $W(\underline{z},n-1,(n-1),s)$, the set $W(\underline{z},n-1,(n-1),s)$ is infinite, and therefore the decomposition of $\Ext^j_S(J_{(\underline{z},l)},S)$ in this case is infinite. So by Lemma \ref{length=dim} again we have $\ell(\Ext^j_S(J_{(\underline{z},l)},S)) = \Dim_\mathbb{C}(\Ext^j_S(J_{(\underline{z},l)},S)) = \infty$. 
	Therefore $\ell(\Ext^j_S(S/I^d,S)) < \infty$ if and only if $j=n(m-n)+1$.
\end{proof}

\begin{corollary}\label{sum_of_length}
	Let $j=n(m-n)+1$. Then we have $$\ell(\Ext^j_S(S/I^D,S)) = \sum^{D}_{d=n}\ell(\Ext^j_S(I^{d-1}/I^d,S)).$$
\end{corollary}
\begin{proof}
	Given the short exact sequence $$0\rightarrow I^{d-1}/I^d \rightarrow S/I^d \rightarrow S/I^{d-1} \rightarrow 0$$ we have the induced long exact sequence of $\Ext$-modules 
	\begin{align*}
		... \rightarrow &\Ext^{j-1}_S(I^{d-1}/I^d,S) \rightarrow \Ext^j_S(S/I^{d-1},S) \rightarrow \Ext^j_S(S/I^d,S)\\ \rightarrow &\Ext^j_S(I^{d-1}/I^d,S) \rightarrow \Ext^{j+1}_S(S/I^{d-1},S) \rightarrow ...
	\end{align*}
	By  Proposition \ref{isom_ext_modules} and Lemma \ref{injectivity} the map $\Ext^j(S/I^{d-1},S) \rightarrow \Ext^j(S/I^d,S)$ from the above long exact sequence is injective. Therefore we can split the above long exact sequence into short exact sequences $$0\rightarrow \Ext^j_S(S/I^{d-1},S) \rightarrow \Ext^j_S(S/I^d,S) \rightarrow \Ext^j_S(I^{d-1}/I^d,S) \rightarrow 0.$$ 
	By Lemma 3.3, $\Ext^j_S(S/I^d,S) = 0$ for $d<n$, so $\Ext^j_S(I^{d-1}/I^d) = 0$ for $d<n$ as well. Then by Proposition \ref{length_ext_fin} we have 
	\begin{align*}
		\begin{gathered}
			\ell(\Ext^j_S(I^{d-1}/I^d,S)) = \ell(\Ext^j_S(S/I^d,S)) - \ell(\Ext^j_S(S/I^{d-1},S)) \implies  \\
			\sum^D_{d=2} \ell(\Ext^j_S(I^{d-1}/I^d,S) = \ell(\Ext^j_S(S/I^D,S))-\ell(\Ext^j_S(S/I,S) \xRightarrow{\text{Lemma 3.3}}\\
			\sum^D_{d=n}\ell(\Ext^j_S(I^{d-1}/I^d,S)) = \ell(\Ext^j_S(S/I^D,S)),
		\end{gathered}
	\end{align*}
	as desired.
\end{proof}

\section{Multiplicites of the local cohomology of determinantal thickenings}
In this section we will prove the main result of this note. We recall the statement here.

\begin{theorem}\label{formula_limit}
	\begin{enumerate}
		\item If $j\neq n^2-1$, then $\ell(H^j_\frakm(R/I_n^D))$ is either $0$ or $\infty$. 
		\item If $j=n^2-1$, then $\ell(H^j_\frakm(R/I_n^D)) < \infty$ and is nonzero. Moreover we have $$\epsilon^j(I_n) = (mn)!\prod^{n-1}_{i=0}\dfrac{i!}{(m+i)!},$$ which is rational.
	\end{enumerate}
\end{theorem}

We will first prove the existence of the limit in question.
\begin{prop}\label{existence_limit}
	Let $$C=(mn-1)!\prod_{1\leq i \leq n}\dfrac{1}{(n-i)!(m-i)!}$$ and let $\delta=\{0\leq x_{n-1} \leq ... \leq x_1 \leq 1\}\subseteq \mathbb{R}^{n-1}$, $dA=dx_{n-1}...dx_1$. Then $\ell(H^{n^2-1}_\frakm(R/I^D)) < \infty$ and is nonzero for $D\geq n$. Moreover the limit $\epsilon^{n^2-1}(I) = \lim_{D\rightarrow \infty}\dfrac{(mn)!\ell(H^{n^2-1}_\frakm(S/I^D)}{D^{mn}}$ exists and the formula is given by
	\begin{align}\label{integral_formula}
			 C\bigintss_\delta(\prod_{1\leq i \leq n-1}(1-x_i)^{m-n}x_i^2)(\prod_{1\leq i < j\leq n-1}(x_i-x_j)^2)dA.
\end{align}
\end{prop}

Before we give the proof of the above Proposition, we need to state some results. We will use the local duality to study $\ell(H^j_\frakm(S/I^D))$. Let $M^\vee$ denote the graded Matilis dual of an $R$-module $M$ where $R$ is a polynomial ring over $\mathbb{C}$ such that $$(M^\vee)_\alpha := \Hom_\mathbb{C}(M_-\alpha,\mathbb{C}).$$

\begin{lemma}\label{isom_of_length}
	Let $J$ be any ideal of a ring $S$. We have $\ell(\Ext^j_S(S/J^D,S)) = \ell(H^{\Dim(S)-j}_\frakm(S/J^D))$.
\end{lemma}
\begin{proof}
	By the local duality (see \cite[theorem 3.6.19]{BrunsHerzog}), we have $$\Ext^j_S(S/J^D,S(-\Dim(S))) \cong H^{\Dim(S)-j}_\frakm(S/J^D)^\vee.$$ Then the assertion of our lemma is immediate.
\end{proof}
Using this lemma we turn the problem into studying the length of $\Ext$-modules of cohomological degree $n(m-n)+1$. 
In the proof of Theorem \ref{existence_limit}, We will employ part of the strategy used in \cite{Kenkel}. However we will not resort to binomial coefficients since they will be too complicated to study in higher dimensional rings.

\begin{theorem}(Euler-Maclaurin formula, see \cite{Apostol})\label{EM}
	Suppose $f$ is a function with continuous derivative on the interval $[1,b]$, then 
	$$\sum^b_{i=a}f(i) = \int^b_af(x)dx+\dfrac{f(b)+f(a)}{2}+\sum^{\lfloor p/2 \rfloor}_{k=1}\dfrac{B_{2k}}{(2k)!}(f^{(2k-1)}(b)-f^{(2k-1)}(a))+R_p$$ where $B_{2k}$ is the Bernoulli number and $R_p$ is the remaining term.
\end{theorem}
For our application we only need to use the integral part on the RHS of the above formula. A well-known consequence is the following.
\begin{corollary}(Faulhaber's formula)\label{Faulhaber}
	The closed formula of the sum of $p$-th power of the first $b$ integers can be written as 
	\begin{align*}
		\sum^b_{k=1}k^p = \dfrac{b^{p+1}}{p+1}+\dfrac{1}{2}b^p+\sum^p_{k=2}\dfrac{B_k}{k!}\dfrac{p!}{p-k+1!}b^{p-k+1}.
	\end{align*}
Again the $B_k$ is the Bernoulli number. In particular, the sum on the LHS can be expressed as a polynomial of degree $p+1$ in $b$ with leading coefficient $\dfrac{1}{p+1}$.
\end{corollary}
\begin{proof}[Proof of Proposition \ref{existence_limit}]
	Let $s$ be as in Theorem \ref{schur-decompos}. By Lemma \ref{vanishing_degree}, we have $\Ext^j_S(S/I^D,S)\neq 0$ for $s=0$, so $\epsilon^j(I) \neq 0$. Again the first claim follows from Proposition \ref{length_ext_fin} and Lemma \ref{isom_of_length}. We will prove the second claim. We first consider the length of $\Ext^j_S(I^{d-1}/I^d,S)$. By Corollary \ref{sum_of_length}, we need to calculate each $\ell(\Ext^j_S(I^{d-1}/I^d,S))$. By Lemma \ref{injectivity}, in order to calculate $\ell(\Ext^j_S(I^{d-1}/I^d,S)$ we only need to calculate the dimension of the tensor products of Schur modules that is in $\Ext^j_S(S/I^{d-1})$ but not in $\Ext^j_S(S/I^d,S)$. By Lemma \ref{weights_max_minors}, we need to consider the $\underline{z}$ such that $\{z_1=...=z_n=d-1\}$. This means we are considering the weights 
\begin{align*}
	\begin{cases}
		\lambda_n=n-d-m,\\
		\lambda_1\leq -m.
	\end{cases}
\end{align*}
Adopting the notations of \cite{Kenkel}, we can write 
\begin{align*}
	\lambda &= (\lambda_1,\lambda_2,...\lambda_n)\\
		&= (\lambda_n+\epsilon_1,\lambda_n+\epsilon_2,...,\lambda_n)\\
		&= (n-d-m+\epsilon_1, n-d-m+\epsilon_2,...,n-d-m).
\end{align*}
Since $\lambda_1\leq -m$, it follows that $n-d-m \leq n-d-m+\epsilon_1 \leq m \implies 0\leq \epsilon_1 \leq d-n$. Since $\lambda$ is dominant, we have $0\leq \epsilon_{n-1}\leq ... \leq \epsilon_1 \leq d-n$. By Corollary \ref{same_dim}, we have $$\Dim(S_\lambda\mathbb{C}^n) = \Dim(S_{(\epsilon_1,...,\epsilon_{n-1},0)}\mathbb{C}^n$$ by adding $((n-d-m)^n)$ to $\lambda$. Therefore the dimension of $S_\lambda \mathbb{C}^n$ is given by 
\begin{align}\label{dim_small}
	\Dim(S_\lambda \mathbb{C}^n) = \Dim(S_{(\epsilon_1,...,\epsilon_{n-1},0)}\mathbb{C}^n) = (\prod_{1\leq i < j\leq n-1}\dfrac{\epsilon_i - \epsilon_j + j-i}{j-i})(\prod_{1\leq i \leq n-1}\dfrac{\epsilon_i+n-i}{n-i}).
\end{align}

Now we look at $S_{\lambda(0)}\mathbb{C}^m$. By definition $\lambda(0)=((-m)^{m-n},\lambda_1,...,\lambda_n)$. Use Corollary \ref{same_dim} again by adding $((n-d-m)^m)$ to $\lambda(0)$ we get that
\begin{align}\label{dim_large}
	\begin{split}
		\Dim(S_{\lambda(0)}\mathbb{C}^m) &= \Dim(S_{((d-n)^{(m-n)},\epsilon_1,...,\epsilon_{n-1},0)}\mathbb{C}^m\\
						 &= (\prod_{1\leq i\leq n-1}\dfrac{j-i}{j-i})(\prod_{1\leq i \leq m-n}\dfrac{d-\epsilon_1+m-2n+1-i}{m-n+1-i})\\
						 & (\prod_{1\leq i\leq m-n}\dfrac{d-\epsilon_2+m-2n+2-i}{m-n+2-i})(\epsilon_1-\epsilon_2+1)\\
						 &...\\
						 & (\prod_{1\leq i\leq m-n}\dfrac{d-n+m-i}{m-i})(\epsilon_{n-1}+1)(\dfrac{\epsilon_{n-2}+2}{2})...(\dfrac{\epsilon_1+n-1}{n-1}).
	\end{split}
\end{align}
Multiplying (\ref{dim_small}) and (\ref{dim_large}) we get that
\begin{align}\label{dim_multiply}
	\begin{split}
		\Dim(S_{\lambda(0)}\mathbb{C}^m \otimes S_\lambda\mathbb{C}^n)
		&= \Dim(S_{\lambda(0)}\mathbb{C}^m) \times \Dim(S_\lambda\mathbb{C}^n) \\
		&= \big(\prod_{1\leq i \leq m-n}\dfrac{d-n+m-i}{m-i}\big)\\
		&\Big(\big(\prod_{1\leq i \leq m-n}\dfrac{d-\epsilon_1-2n+m+1-i}{m-n+1-i}\big)\big(\dfrac{\epsilon_1+n-1}{n-1}\big)^2\\
		&\big(\prod_{1\leq i \leq m-n}\dfrac{d-\epsilon_2-2n+m+2-i}{m-n+2-i}\big)\big(\epsilon_1-\epsilon_2+1)^2\big(\dfrac{\epsilon_2+n-2}{n-2}\big)^2\\
		&...\\
		&\big(\prod_{1\leq i\leq m-n}\dfrac{d-n-\epsilon_{n-1}+m-1-i}{m-1-i}\big)(\epsilon_{n-2}-\epsilon_{n-1}+1)^2...(\epsilon_{n-1}+1)^2\Big)
	\end{split}
\end{align}
The formula (\ref{dim_multiply}) is for a particular choice of $\epsilon_1,...,\epsilon_{n-1}$. To calculate $\ell(\Ext^j_S(I^{d-1}/I^d,S))$ we need to add the result of all possible choices of $\epsilon_1,...,\epsilon_{n-1}$. After some manipulations we will end up with 
\begin{align}\label{dim_formula_sep}
	\ell(\Ext^j_S(I^{d-1}/I^d,S) &= \sum_{0\leq \epsilon_{n-1}\leq ... \leq \epsilon_1\leq d-n}(\ref{dim_multiply})\\
	\tag{\ref{dim_formula_sep} - 0} &= (\prod_{1\leq i\leq m-n}\dfrac{d-n+m-i}{m-i}) \\ 
	\tag{\ref{dim_formula_sep} - 1}	     &\Big(\sum^{d-n}_{\epsilon_1=0}\big(\prod_{1\leq i \leq m-n}\dfrac{d-\epsilon_1-2n+m+1-i}{m-n+1-i}\big)\big(\dfrac{\epsilon_1+n-1}{n-1}\big)^2\\
	\tag{\ref{dim_formula_sep} - 2}			     &(\sum^{\epsilon_1}_{\epsilon_2=0}\big(\prod_{1\leq i\leq m-n}\dfrac{d-\epsilon_2-2n+m+2-i}{m-n+2-i}\big)(\epsilon_1-\epsilon_2+1)^2\big(\dfrac{\epsilon_2+n-2}{n-2}\big)^2\\
	\notag					     &...\\
\tag{\ref{dim_formula_sep} - (n-2)}	     &(\sum^{\epsilon_{n-3}}_{\epsilon_{n-2}=0}\big(\prod_{1\leq i\leq m-n}\dfrac{d-n-\epsilon_{n-2}+m-2-i}{m-2-i}\big)(\epsilon_{n-3}-\epsilon_{n-2}+1)^2...(\dfrac{\epsilon_{n-2}+2}{2})^2\big)\\
	\tag{\ref{dim_formula_sep} - (n-1)}	     &(\sum^{\epsilon_{n-2}}_{\epsilon_{n-1}=0}\big(\prod_{1\leq i\leq m-n}\dfrac{d-n-\epsilon_{n-1}+m-1-i}{m-1-i}\big)(\epsilon_{n-2}-\epsilon_{n-1}+1)^2...(\epsilon_{n-1}+1)^2\big)...\Big)
\end{align}
Now Corollary \ref{Faulhaber} shows that the above sum will be a polynomial in $d$, and we need to calculate its degree. Corollary \ref{Faulhaber} also implies that when looking at each sum of (\ref{dim_formula_sep}) we only need to look at the summands that will contribute to the highest degree of the resulting polynomial. We see that the sum (\ref{dim_formula_sep} - (n-1)) can be expressed as a degree $m-n+2(n-1)+1 = m+n-1$ polynomial in $\epsilon_{n-2}$. Similarly (\ref{dim_formula_sep} - (n-2)) can be expressed as a degree $2m+2n-4$ polynomial in $\epsilon_{n-3}$. Continuing in this fashion we see that the sum (\ref{dim_formula_sep} - 1) can be expressed as a degree $mn-m+n-1$ polynomial in $d$. Multiplying (\ref{dim_formula_sep} - 0) with (\ref{dim_formula_sep} - 1) will result in a degree $mn-1$ polynomial.

Moreover, after factoring out the coefficients of the terms that will eventually contribute to the highest degree of the resulting polynomial of (\ref{dim_formula_sep}) and then apply Theorem \ref{EM} to the sum of said terms, the leading coefficient of the resulting polynomial of (\ref{dim_formula_sep}) is given by  
\begin{align*}
		\begin{gathered}
			\prod_{1\leq i\leq n}\dfrac{1}{(n-i)!(m-i)!} \lim_{d\rightarrow \infty}\dfrac{\bigintss_\Delta(\prod_{1\leq i \leq n-1}(d-x_i)^{m-n})(\prod_{1\leq i \leq n-1}x_i^2)(\prod_{1\leq i < j\leq n-1}(x_i-x_j)^2)dA}{d^{mn-m+n-1}},\\
			\Delta = \{0 \leq x_{n-1} \leq ... \leq x_1 \leq d-n \}.
\end{gathered}
\end{align*}
where the factor $\frac{1}{(m-1)!(n-1)!}$ comes from $(\ref{dim_formula_sep} - 0)$ and the coefficients of $(\dfrac{\epsilon_i}{n-i})^2$, and the product $\prod_{2\leq i\leq n}\frac{1}{(n-i)!(m-i)!}$ comes from the coefficients of the needed terms from the rest of $(\ref{dim_formula_sep})$. Since the above limit exists and the integrand is a homogeneous polynomial in $d,x_1,...,x_{n-1}$, we can simplify it to
\begin{align}\label{formula_mid}
		\begin{gathered}
			\prod_{1\leq i\leq n}\dfrac{1}{(n-i)!(m-i)!} \bigintss_\delta(\prod_{1\leq i \leq n-1}(1-x_i)^{m-n}x_i^2)(\prod_{1\leq i < j\leq n-1}(x_i-x_j)^2)dA,\\
			\delta = \{0 \leq x_{n-1} \leq ... \leq x_1 \leq 1 \}.
\end{gathered}
\end{align}

By Corollary \ref{sum_of_length} we need to sum $\ell(\Ext^j_S(I^{d-1}/I^d,S))$ over all $n \leq d\leq D$ to get $\ell(\Ext^j_S(S/I^D,S))$. It is clear that by Corollary \ref{Faulhaber} the sum 
\begin{align}\label{sum_all_d}
\ell(\Ext^j_S(S/I^D,S))=\sum^D_{d=n}\ell(\Ext^j_S(I^{d-1}/I^d,S))
\end{align}
can be expressed as a polynomial in $D$ of degree $mn$. By Lemma \ref{isom_of_length} we see that $\ell(H^{mn-j}_\frakm(S/I^D))$ is a polynomial in $D$ of degree $mn$ as well. Therefore we have $$\epsilon^{mn-j}(I) = \lim_{D\rightarrow \infty}\dfrac{(mn)!\ell(H^{mn-j}_\frakm(S/I^D))}{D^{mn}} < \infty,$$ where $mn-j = mn-n(m-n)-1 = n^2-1$.

Finally, apply Corollary 4.4 to (\ref{sum_all_d}), we see that the leading coefficient of the resulting polynomial of (\ref{sum_all_d}) is given by multiplying $1/mn$ to $(\ref{formula_mid})$, then multiplying the result with $(mn)!$ yields the desired formula. 
\end{proof}

We will give some examples of the above formula.

\begin{example}
	Let $n=2$ and $j=3$. By Lemma \ref{vanishing_degree} and Lemma \ref{length_ext_fin}, $H^3_\frakm(S/I^D)\neq 0$ and has finite length. The integral we need to calculate is simply $$\int^{1}_{0}(1-x_1)^{m-2}x_1^2dx_1 = \dfrac{2}{m^3-m}.$$ Since $C=\frac{(2m-1)!}{(m-1)!(m-2)!}$, we get that $$\epsilon^3(I) = \dfrac{1}{(m+1)!m!}(2m)! = \dfrac{1}{m+1}\binom{2m}{m}.$$ This recovered the result from \cite[Corollary 1.2]{Kenkel}.
\end{example}

\begin{example}
	Let $n=3$ and $j=8$. Again one can check with Lemma \ref{vanishing_degree} and \ref{length_ext_fin} that $H^8_\frakm(S/I^D)\neq 0$ and has finite length. By Proposition \ref{existence_limit} we first calculate the integral $$\bigintss_{0\leq x_2 \leq x_1 \leq 1}(1-x_1)^{m-3}(1-x_2)^{m-3}x_1^2x_2^2(x_1-x_2)^2 dx_2dx_1.$$ This can be done by doing integration by parts multiple times or simply use Sage. The result is $\frac{12}{m^2(m^2-4)(m^2-1)^2}$. 

	We also have $C=(3m-1)!\frac{1}{(m-3)!}\frac{1}{(m-2)!}\frac{1}{2(m-1)!}$. Therefore 
\begin{align*}
	\epsilon^8(I) &= (3m-1)!\dfrac{12}{m^2(m^2-4)(m^2-1)^2(m-3)!(m-2)!2(m-1)!}\\
		      &= (3m)!\dfrac{2}{(m+2)!(m+1)!m!}.
\end{align*}
\end{example}

The examples above hinted that $\epsilon^{n^2-1}(I)$ should be $(mn)!\prod^{n-1}_{i=0}\dfrac{i!}{(m+i)!}$ as stated in the main theorem, and we will prove that this is indeed the case. We first recall a classical result of Atle Selberg. For English reference one might check \cite[(1.1)]{ForresterWarnaar}.
\begin{theorem}(See \cite{Selberg})\label{Selberg_int_thm} 
	For $a, b$ and $c$ in $\mathbb{C}$ such that $\operatorname{Re}(a) > 0$, $\operatorname{Re}(b) > 0$ and $\operatorname{Re}(c) > -\operatorname{min}\{1/n, \operatorname{Re}(a)/(n-1), \operatorname{Re}(b)/(n-1)\}$ we have 
	\begin{align*}\label{Selberg_int}
		S_n(a,b,c) &= \bigintss_{[0,1]^n}\prod^n_{i=1}x_i^{a-1}(1-x_i)^{b-1}\prod_{1\leq i < j \leq n}|x_i-x_j|^{2c}dA \\
			   &= \prod^{n-1}_{i=0} \dfrac{\Gamma(a+ic)\Gamma(b+ic)\Gamma(1+(i+1)c)}{\Gamma(a+b+(n+i-1)c)\Gamma(1+c)}.
	\end{align*}
	where $\Gamma$ is the usual Gamma function $\Gamma(k)=(k-1)!$.
\end{theorem}
Now we can prove Theorem \ref{formula_limit}.

\begin{proof}[Proof of Theorem \ref{formula_limit}]	
	(1) Follows from Lemma \ref{vanishing_degree}, Proposition \ref{length_ext_fin} and Lemma \ref{isom_of_length}.

	(2) By Proposition \ref{existence_limit} it remains to evaluate $$C\bigintss_\delta \prod^{n-1}_{i=1}(1-x_i)^{m-n}x_i^2\prod_{1\leq i < j \leq n-1}(x_i-x_j)^2 dA$$ where $$C = (mn-1)!\prod^n_{i=1}\dfrac{1}{(m-i)!(n-i)!},  \delta = \{0 \leq x_{n-1} \leq ... \leq x_1 \leq 1\}.$$

	By Theorem \ref{Selberg_int_thm} we have 
	\begin{align*}
		& C\bigintss_{[0,1]^{n-1}}\prod^{n-1}_{i=1}(1-x_i)^{m-n}x_i^2\prod_{1\leq i < j \leq n-1}(x_i-x_j)^2 dA
	     \\ &= C \prod^{n-2}_{i=0}\dfrac{\Gamma(3+i)\Gamma(m-n+1+i)\Gamma(2+i)}{\Gamma(m+i+2)\Gamma(2)} \\ &= C \prod^{n-2}_{i=0}\dfrac{(2+i)!(m-n+i)!(1+i)!}{(m+i+1)!}
	     \\ &= \dfrac{(mn)!}{mn}\dfrac{1}{(m-n)!}\prod^{n-1}_{i=1}\dfrac{1}{(m-i)!(n-i)!}\prod^{n-1}_{i=1}\dfrac{(1+i)!(m-n+i-1)!(i)!}{(m+i)!}
	     \\ &= \dfrac{(mn)!}{mn}\dfrac{1}{(m-n)!}\prod^{n-1}_{i=1}\dfrac{(1+i)!}{(m-n+i)!(m-n+i)...(m+i)}
	     \\ &= \dfrac{(mn)!}{mn}\dfrac{1}{(m-n)!}\dfrac{(m-n)!}{(m-1)!}\prod^{n-1}_{i=1}\dfrac{(1+i)!}{(m+i)!}
	     \\ &= \dfrac{(mn)!}{n}\prod^{n-1}_{i=0}\dfrac{1}{(m+i)!}\prod^{n-1}_{i=1}(1+i)!
	\end{align*}
	Since the integrand $\prod^{n-1}_{i=1}(1-x_i)^{m-n}x_i^2\prod_{1\leq i < j \leq n-1} (x_i-x_j)^2$ does not change under permutation of variables, we have 
	\begin{align*}
		& \bigintss_{[0,1]^{n-1}}\prod^{n-1}_{i=1}(1-x_i)^{m-n}x_i^2\prod_{1\leq i < j \leq n-1}(x_i-x_j)^2 dA 
		\\ = & (n-1)!\bigintss_\delta \prod^{n-1}_{i=1}(1-x_i)^{m-n}x_i^2\prod_{1\leq i < j \leq n-1}(x_i-x_j)^2 dA
	\end{align*}
	Hence we have 
	\begin{align*}
		 \epsilon^{n^2-1}(I) & = C\bigintss_\delta \prod^{n-1}_{i=1}(1-x_i)^{m-n}x_i^2\prod_{1\leq i < j \leq n-1}(x_i-x_j)^2 dA \\
		& = \dfrac{1}{(n-1)!}\dfrac{(mn)!}{n}\prod^{n-1}_{i=0}\dfrac{1}{(m+i)!}\prod^{n-1}_{i=1}(1+i)! \\
		& = (mn)!\prod^{n-1}_{i=0}\dfrac{i!}{(m+i)!}.
	\end{align*}
\end{proof}

\section{Acknowledgement}
The author would like to thank her advisor, Wenliang Zhang, for suggesting this problem and his guidance throughout the preparation of this paper, and Tian Wang for helpful suggestions.

\end{document}